\documentclass[14pt,reqno,a4paper]{amsart}
\usepackage{hyperref}
\usepackage[left=1.25in, right=1.25in]{geometry}
\usepackage{enumerate}
\usepackage{amsmath,amssymb,amsthm,mathrsfs,amsfonts,dsfont}
\usepackage{graphicx}
\usepackage{tikz}
\usetikzlibrary{graphs,graphs.standard}
\usetikzlibrary{shapes}\usetikzlibrary{plotmarks}\usetikzlibrary{arrows}\usetikzlibrary{positioning}
\theoremstyle{definition}
\newtheorem{defn}{Definition}[section]

\newtheorem{prop}[defn]{Proposition}

\newtheorem{Remark}[defn]{Remark}

\newtheorem{thm}[defn]{Theorem}
\newtheorem{corr}[defn]{Corollary}
\newtheorem{lem}[defn]{Lemma}

\newtheorem{question}[defn]{Question}
\newtheorem{remark}[defn]{Remark}

\title[{\em M\MakeLowercase{aximal independent sets}, V\MakeLowercase{ariants of} CAC, \MakeLowercase{and} CWF}]{Maximal independent sets, variants of chain/antichain principle and cofinal subsets without AC}

\author{A\MakeLowercase{mitayu} B\MakeLowercase{anerjee}}

\address{Department of Logic, Institute of Philosophy, E\"otv\"os Lor\'and University,
M\'{u}zeum krt. 4/i Budapest, H-1088 Hungary}
\email{banerjee.amitayu@gmail.com}
\begin{document}
\maketitle
\begin{abstract}
In set theory without the Axiom of Choice (AC), we observe new relations of the following statements with weak choice principles.

\begin{itemize}
     \item $\mathcal{P}_{lf,c}$ (Every locally finite connected graph has a maximal independent set).
    \item $\mathcal{P}_{lc,c}$ (Every locally countable connected graph has a maximal independent set).
    \item $CAC^{\aleph_{\alpha}}_{1}$ (If in a partially ordered set all antichains are finite and all chains have size $\aleph_{\alpha}$, then the set has size $\aleph_{\alpha}$) if $\aleph_{\alpha}$ is regular.
    \item CWF (Every partially ordered set has a cofinal well-founded subset).
    \item $\mathcal{P}_{G,H_{2}}$ (For any infinite graph $G=(V_{G}, E_{G})$ and any finite graph $H=(V_{H}, E_{H})$ on 2 vertices, if every finite subgraph of $G$ has a homomorphism into $H$, then so has $G$).
    \item If $G=(V_{G},E_{G})$ is a connected locally finite chordal graph, then there is an ordering $<$ of $V_{G}$ such that $\{w < v : \{w,v\} \in E_{G}\}$ is a clique for each $v\in V_{G}$.
\end{itemize}
\end{abstract}
\section{introduction}
As usual, ZF denotes the Zermelo-Fraenkel set theory without the Axiom of Choice (AC), and ZFA is ZF with the axiom of extensionality weakened to allow the existence of atoms. In this note, we observe new relations of some combinatorial statements with certain weak forms of AC. Complete definitions of the choice forms will be given in \textbf{Definition 2.4}.

\subsection{Maximal independent sets} Friedman [\cite{Fri2011}, \textbf{Theorem 6.3.2, Theorem 2.4}] proved that $AC$ is equivalent to the statement {\em `Every graph has a maximal independent set'} (abbreviated here as $\mathcal{P}$) in ZF. Spanring \cite{Spa2014} gave a different argument to prove the result.
Consider the following weaker formulations of $\mathcal{P}$. 

\begin{itemize}
\item Fix $n\in \omega\backslash\{0,1\}$. We denote by $P_{K_{n}}$, the class of those graphs whose only components are $K_{n}$ (complete graph on $n$ vertices).
We denote by $\mathcal{P}_{n}$ the statement {\em `Every graph from the class $P_{K_{n}}$ has a maximal independent set'}. 
\item We denote by $\mathcal{P}_{lf,c}$ the statement {\em `Every locally finite connected graph has a maximal independent set'}.
\item We denote by $\mathcal{P}_{lc,c}$ the statement {\em `Every locally countable connected graph has a maximal independent set'}.
\end{itemize}

In this note, we observe the following.
\begin{enumerate}
    \item $AC_{n}$ is equivalent to $\mathcal{P}_{n}$ for every $n\in \omega\backslash\{0,1\}$ in ZF (cf. [\textbf{$\S$3}, \textbf{Proposition 3.2}]).
    
    \item $AC_{fin}^{\omega}$ is equivalent to $\mathcal{P}_{lf,c}$ in ZF (cf. [\textbf{$\S$3}, \textbf{Proposition 3.3}]).
    \item $UT(\aleph_{0},\aleph_{0},\aleph_{0})$  implies $\mathcal{P}_{lc,c}$, and $\mathcal{P}_{lc,c}$ implies $AC_{\aleph_{0}}^{\aleph_{0}}$ in ZF (cf. [\textbf{$\S$3}, \textbf{Proposition 3.4}]).
\end{enumerate}
 
\subsection{A variant of Chain/Antichain principle} A famous application of the infinite Ramsey's theorem is the {\em Chain/Antichain principle} (abbreviated here as ``CAC''), which states that {\em `Any infinite partially ordered set contains either an infinite chain or an infinite antichain'}. Tachtsis \cite{Tac2016} investigated the possible placement of CAC in the hierarchy of weak choice principles. Komj\'{a}th--Totik \cite{KT2006} proved the following generalized versions of CAC, applying Zorn's lemma.
\begin{itemize}
    \item {\em If in a partially ordered set all antichains are finite and all chains are countable, then the set is countable} (cf. [\cite{KT2006}, \textbf{Chapter 11}, \textbf{Problem 8}]). 
    \item {\em If in a partially ordered set all chains are finite and all antichains are countable, then the set is countable} (cf. [\cite{KT2006}, \textbf{Chapter 11}, \textbf{Problem 7}]). 
\end{itemize}

For each regular $\aleph_{\alpha}$, we denote by $CAC^{\aleph_{\alpha}}_{1}$ the
statement {\em `if in a partially ordered set all antichains are finite and all chains have size $\aleph_{\alpha}$, then the set has size $\aleph_{\alpha}$'} and we denote by $CAC^{\aleph_{\alpha}}$ the
statement {\em `if in a partially ordered set all chains are finite and all antichains have size $\aleph_{\alpha}$, then the set has size $\aleph_{\alpha}$'}. In \cite{BG2020}, we observed that for any regular $\aleph_{\alpha}$ and any $2\leq n<\omega$, $CAC^{\aleph_{\alpha}}$ does not imply $AC_{n}^{-}$ in ZFA. In \cite{BG2020}, we also observed that $CAC^{\aleph_{\alpha}}$ does not imply {\em `there are no amorphous sets'} in ZFA. 
In this note, we observe the following.
\begin{enumerate}
    \item Let $n \in \omega \backslash \{0, 1\}$. The statement “For every regular $\aleph_{\alpha}$, $CAC^{\aleph_{\alpha}}_{1}$” implies neither $AC_{n}^{-}$ nor “there are no amorphous sets” in ZFA (cf. [\textbf{$\S$4}, \textbf{Theorem 4.3}]). 
    \item $CAC^{\aleph_{0}}_{1}$ implies  $PAC^{\aleph_{1}}_{fin}$ (Every $\aleph_{1}$-sized family $\mathcal{A}$ of non-empty finite sets has an $\aleph_{1}$-sized subfamily $\mathcal{B}$ with a choice function) in ZF (cf. [\textbf{$\S$4}, \textbf{Theorem 4.5}]).
    \item $DC$ does not imply $CAC_{1}^{\aleph_{0}}$ in ZF (cf. [\textbf{$\S$4}, \textbf{Corollary 4.6}]).
\end{enumerate}

\subsection{Cofinal well-founded subsets and consistency results}  Halbeisen--Tachtsis [\cite{HT2020}, \textbf{Theorem 10(ii)}] constructed a model of ZFA and proved that LOC$_{2}^{-}$ does not imply $LOKW_{4}^{-}$ in ZFA. 
We construct a similar model of ZFA and observe the following.

\begin{enumerate}
\item ($LOC_{2}^{-}$ + $MC$) does not imply $LOC_{n}^{-}$ in ZFA if $n\in\omega$ such that $n=3$ or $n>4$ (cf. [\textbf{$\S$5}, \textbf{Theorem 5.3}]). 
\item ($LOC_{2}^{-}$ + $MC$) does not imply $CAC^{\aleph_{0}}_{1}$ in ZFA (cf. [\textbf{$\S$5}, \textbf{Corollary 5.4}]).
\end{enumerate}

We also observe that under certain hypotheses on the group $\mathcal{G}$ and the normal filter $\mathcal{F}$, $CWF$ and $CS$ are true in the resulting permutation model $\mathcal{N}$ (cf. \textbf{Lemma 5.1}). 

\subsection{A generalized formulation of the $n$-coloring theorem} 
Fix a natural number $n\in \omega\backslash\{0,1\}$. Komj\'{a}th [\cite{Kom}, \textbf{Theorem 4.5.2}] sketched a proof of the following generalization of the {\em $n$-coloring theorem} applying $BPI$: {\em `For any infinite graph $G=(V_{G}, E_{G})$ and any finite graph $H=(V_{H}, E_{H})$, if every finite subgraph of $G$ has a homomorphism into $H$, then so has $G$'} abbreviated here as  $\mathcal{P}_{G,H}$. We denote by $\mathcal{P}_{G,H_{n}}$ the above statement if $H$ has $n$ vertices for $n\in \omega\backslash\{0,1\}$. Clearly, for every $n\in \omega\backslash\{0,1\}$, $\mathcal{P}_{G,H_{n}}$ implies the $n$-coloring theorem in ZF (consider the finite graph $H$ to be $K_{n}$), and the $n$-coloring theorem implies $AC_{n}$ in ZF \cite{Myc1964}. We note that for any integer $n \geq 3$, the $n$-coloring theorem is equivalent to $BPI$ in ZF, as shown by L\"{a}uchli \cite{Leau1971}. Consequently, $\mathcal{P}_{G, H_{n}}$ is equivalent to $BPI$ in ZF for every integer $n\geq 3$. In \cite{BG2020}, we observed that if $X\in \{AC_{3}, AC_{fin}^{\omega}\}$, then $\mathcal{P}_{G,H_{2}}$ does not imply $X$ in ZFA.
In this note, we observe that $AC_{2}$ is equivalent to $\mathcal{P}_{G,H_{2}}$ in ZF (cf. [\textbf{$\S$3}, \textbf{Proposition 3.6}]). 
 
\subsection{Locally finite connected graphs}   
Locally finite connected graphs are studied extensively in graph theory (cf. \cite{Die2017}). We list some graph-theoretical statements restricted to locally finite connected graphs, which follows from $AC^{\omega}_{fin}$ in ZF (cf. [\textbf{$\S$3}, \textbf{Remark 3.9}]). Moreover, we prove the following. 

\begin{enumerate}
    \item $AC_{fin}^{\omega}$ implies $\mathcal{P}_{G,H}$ in ZF, if $G$ is a locally finite connected graph (cf. [\textbf{$\S$3}, \textbf{Proposition 3.7}]).
    \item $AC_{fin}^{\omega}$ implies the statement {\em `If $G=(V_{G},E_{G})$ is a connected locally finite chordal graph, then there is an ordering $<$ of $V_{G}$ such that $\{w < v : \{w,v\} \in E_{G}\}$ is a clique for each $v\in V_{G}$'} in ZF (cf. [\textbf{$\S$3}, \textbf{Proposition 3.8}]).
\end{enumerate} 

%%%%%%%%%%%%%%%%%%%%%%%%%%%%%
\section{Notations, definitions, and known results} 
\begin{defn}{\em \textbf{(Graph-theoretical definitions, and notations)}.} The {\em degree} of a vertex $v\in V_{G}$ of a graph $G=(V_{G}, E_{G})$ is the number of edges emerging from $v$.
A graph $G=(V_{G}, E_{G})$ is {\em locally finite} if every vertex of $G$ has finite degree. We say that a graph $G=(V_{G}, E_{G})$ is {\em locally countable} if for every $v \in V_{G}$, the set of neighbors of $v$ is countable. Given a non-negative integer $n$, a {\em path of length $n$} in the graph $G=(V_{G}, E_{G})$ is a
one-to-one finite sequence $\{x_{i}\}_{0\leq i \leq n}$ of vertices such that for each $i < n$, $\{x_{i}, x_{i+1}\} \in E_{G}$; such a path joins $x_{0}$ to $x_{n}$. The graph $G$ is {\em connected} if any two vertices are joined by a path of finite length. A homomorphism from a graph $G=(V_{G}, E_{G})$ to a graph $H=\{V_{H}, E_{H}\}$ is a map $f$ from $V_{G}$ to $V_{H}$, such that if $\{v_{1}, v_{2}\}\in E_{G}$ then $\{f(v_{1}), f(v_{2})\}\in E_{H}$. A {\em good coloring} of a graph $G=(V_{G},E_{G})$ with a color set $C$ is a mapping $f:V_{G}\rightarrow C$ such that for every $\{x,y\}\in E_{G}$, $f(x)\not= f(y)$. Fix a natural number $n\in \omega$. A graph $G=(V_{G},E_{G})$ is {\em $n$-colorable} if there exists a good coloring of $G$ on $n$ colors.
We denote by $K_{n}$, the complete graph on $n$ vertices. We denote by $C_{n}$ the circuit of length $n$. A graph is {\em chordal} if it does not contain an induced $C_{n}$ for $n \geq 4$.
An {\em independent set} is a set of vertices in a graph, no two of which are connected by an edge. A set $W_{G} \subseteq V_{G}$ is called a {\em maximal independent set} in $G=(V_{G}, E_{G})$ if and only if it is independent and there is no independent set $W'_{G}$ such that $W_{G} \subsetneq W'_{G}$ (cf. \cite{Spa2014}). A {\em clique} is a set of  vertices in a graph, such that any two of them are joined by an edge. 
\end{defn}

\begin{defn} {\em \textbf{(Chain, antichain, cofinal well-founded subsets)}.} Let $(P, \leq)$ be a {\em partially ordered set} or a {\em poset}. 
A subset $D \subseteq P$ is called a {\em chain} if $(D, \leq\restriction D)$ is linearly ordered. A subset $A\subseteq P$ is called an {\em antichain}
if no two elements of $A$ are comparable under $\leq$. A subset $C \subseteq P$ is called {\em cofinal} in $P$ if for every $x \in P$ there is an element $c \in C$ such that $x \leq c$. An element $p \in P$ is {\em minimal} if for all $q \in P$, $(q \leq p)$ implies $(q = p)$. A subset $W \subseteq P$ is {\em well-founded} if every non-empty subset $V$ of $W$ has a $\leq$-minimal element.
\end{defn}

\begin{defn}{\em \textbf{(Amorphous sets).}} An inﬁnite set $X$ is called {\em amorphous} if $X$ cannot be written as a disjoint union of two inﬁnite subsets.
\end{defn}

\begin{defn}{\textbf{\em(A list of choice forms).}}
\begin{enumerate}
\item The \textbf{{\em Axiom of Choice}, $AC$ (Form 1 in \cite{HR1998})}: Every family of nonempty sets has a choice function.
\item  The \textbf{{\em Axiom of Choice for Finite Sets}, $AC_{ﬁn}$ (Form 62 in \cite{HR1998})}: Every family of non-empty ﬁnite sets has a choice function.
\item \textbf{{\em $AC_{fin}^{\omega}$}(Form 10 in \cite{HR1998})}:  Every denumerable, i.e. countably infinite, family of non-empty finite sets has a choice function. We recall two equivalent formulations of $AC_{fin}^{\omega}$.
\begin{itemize}
    \item \textbf{{\em $UT(\aleph_{0}, fin, \aleph_{0})$} (Form 10 A in \cite{HR1998})}: The union of denumerably many pairwise disjoint finite sets is denumerable. 
    \item \textbf{{\em $PAC_{fin}^{\omega}$}(Form 10 E in \cite{HR1998})}: Every denumerable family of finite sets has an infinite
subfamily with a choice function. 
\end{itemize}

\item \textbf{$AC_{\aleph_{0}}^{\aleph_{0}}$ (Form 32 A in \cite{HR1998})}:  Every denumerable family of denumerable sets has a choice function. We recall the following equivalent formulation of $AC_{\aleph_{0}}^{\aleph_{0}}$.
\begin{itemize}
\item  \textbf{$PAC_{\aleph_{0}}^{\aleph_{0}}$ (Form 32 B in \cite{HR1998})}:
Every denumerable set of denumerable sets has an infinite subset with a choice function. 
\end{itemize}

\item \textbf{$AC_{2}$ (Form 88 in \cite{HR1998})}:  Every family of pairs has a choice function. 
\item \textbf{$AC_{n}$ for each $n \in \omega, n \geq 2$ (Form 61 in \cite{HR1998})}: Every family of $n$-element sets has a choice function. We denote by $AC_{n}^{-}$ the statement {\em `Every infinite family $\mathcal{A}$ of $n$-element sets has a partial choice function, i.e., $\mathcal{A}$ has an infinite
subfamily $\mathcal{B}$ with a choice function.'} (cf. \textbf{Form 342(n) in \cite{HR1998}}).

\item \textbf{$LOC_{n}^{-}$ for each $n \in \omega, n \geq 2$} (see \cite{HT2020}): Every infinite linearly orderable family of $n$-element sets has a partial choice function. We denote by $LOKW_{n}^{-}$ the statement {\em `Every infinite linearly orderable family $\mathcal{A}$ of $n$-element sets has a partial Kinna--Wagner selection function, i.e., there exists an infinite subfamily $\mathcal{B}$ of $\mathcal{A}$ and a function $f$ such that $dom(f) = \mathcal{B}$
and for all $B \in \mathcal{B}$, $\emptyset \not= f(B)\subsetneq B$ (f is called a Kinna–Wagner selection function for $\mathcal{B}$).'} (cf. \textbf{Definition 1 (2)} of \cite{HT2020}).

\item \textbf{{\em Van Douwen’s Choice Principle}, $vDCP$} (see \cite{HT2013}): Every family $X = \{(X_{i}, \leq_{i}) : i \in I\}$ of linearly ordered sets
isomorphic with $(\mathbb{Z}, \leq)$ ($\leq$ is the usual ordering on $\mathbb{Z}$) has a choice function.

\item The \textbf{{\em Axiom of Multiple Choice}, $MC$ (Form 67 in \cite{HR1998})}: Every family $\mathcal{A}$ of non-empty sets has a multiple choice function, i.e., there is a function $f$ with domain $\mathcal{A}$ such that for every $A \in \mathcal{A}$, $f(A)$ is a non-empty finite subset of $A$.

\item \textbf{$MC(n)$ where $n \geq 2$ is an integer (see \cite{HT2013})}: For every family $\{X_{i}:i\in I\}$ of non-empty sets, there is a function $F$ with domain $I$ such that for all $i \in I$, we have that $F(i)$ is a finite subset of $X_{i}$ and $gcd(n,\vert F(i)\vert)=1$. 
\item \textbf{$LW$ (Form 90 in \cite{HR1998})}: Every linearly-ordered set can be well-ordered.

\item \textbf{$AC^{LO}$ (Form 202 in \cite{HR1998})}: Every linearly ordered family of non-empty sets has a choice function.
    
\item \textbf{$AC^{WO}$ (Form 40 in \cite{HR1998})}: Every well-ordered family of non-empty sets has a choice function.

\item \textbf{$DC_{\kappa}$ for an inﬁnite well-ordered cardinal $\kappa$ (Form 87($\kappa$) in \cite{HR1998})}: Let $\kappa$ be an inﬁnite well-ordered cardinal (i.e., $\kappa$ is an aleph). Let $S$ be a non-empty set and let $R$ be a binary relation such that for every $\alpha<\kappa$ and every $\alpha$-sequence $s =(s_{\epsilon})_{\epsilon<\alpha}$ of elements of $S$ there exists $y \in S$ such that $s R y$. Then there is a function $f : \kappa \rightarrow S$ such that for every $\alpha < \kappa$, $(f\restriction \alpha) R f(\alpha)$. We note that $DC_{\aleph_{0}}$ is a reformulation of $DC$ (the principle of Dependent Choices \textbf{(Form 43 in \cite{HR1998})}). We denote by $DC_{<\lambda}$ the assertion $(\forall\eta < \lambda)DC_{\eta}$.

\item \textbf{$UT(WO, WO, WO)$ (Form 231 in \cite{HR1998})}: The union of a well-ordered collection of well-orderable sets is well-orderable.
\item \textbf{$(\forall\alpha)UT(\aleph_{\alpha},\aleph_{\alpha},\aleph_{\alpha})$ (Form 23 in \cite{HR1998})}: For every ordinal $\alpha$, if $A$ and every member of $A$ has cardinality $\aleph_{\alpha}$, then $\vert \cup A\vert=\aleph_{\alpha}$.

\item {\em $\aleph_{1}$ is regular} \textbf{(Form 34 in \cite{HR1998})}.
\item The \textbf{{\em Boolean Prime Ideal Theorem}, $BPI$ (Form 14 in \cite{HR1998})}: Every Boolean algebra has a prime ideal. We recall the following equivalent formulation of BPI.
\begin{itemize}
 \item The \textbf{{\em n-coloring theorem for $n\geq 3$}, (Form 14G($n$)($n\in\omega, n\geq 3$) in \cite{HR1998})}: For every graph $G=(V_{G},E_{G})$ if every finite subgraph of $G$ is $n$-colorable then $G$ is $n$-colorable. This is De Bruijn--Erd\H{o}s theorem for $n\geq 3$ colorings.
\end{itemize}

\item \textbf{{\em Marshall Hall's theorem}, $MHT$ (Form 107 in \cite{HR1998})}: If $S$ is a set and $\{S_{i}\}_{i\in I}$ is an indexed family of \textbf{finite} subsets of $S$, then if the following property holds,
    \begin{center}
        (P) for every finite $F\subseteq I$, there is an injective choice function for $\{S_{i}\}_{i\in F}$.
    \end{center}
    then there is an injective choice function for $\{S_{i}\}_{i\in I}$. 

\item \textbf{{\em Dilworth’s  decomposition  theorem for  infinite  posets  of  finite  width,} $DT$ (cf.\cite{Tac2019}):}  If $\mathbb{P}$ is an arbitrary poset, and $k$ is a natural number such that $\mathbb{P}$ has no antichains of size $k + 1$ while at least one $k$-element subset of $\mathbb{P}$ is an antichain, then $\mathbb{P}$ can be partitioned into $k$ chains.  

\item \textbf{Rado’s Selection Lemma, $RSL$ (Form 99 in \cite{HR1998})}: Let $\mathcal{F}$ be a family of finite sets and suppose that to every finite subset $F$ of $\mathcal{F}$ there corresponds a choice function $\phi_{F}$ whose domain is $F$ such that $\phi_{F}(T) \in T$ for each $T \in F$. Then there is a choice function $f$ whose domain is $\mathcal{F}$ with the property that for every finite subset $F$ of $\mathcal{F}$, there is a finite subset $F'$ of $\mathcal{F}$ such that $F \subseteq F'$ and $f(T) = \phi_{F'}(T)$ for all $T \in F$.

\item The \textbf{{\em Antichain Principle} (Form 89 in \cite{HR1998})}: Every partially ordered set has a maximal antichain.

\item The \textbf{{\em Chain/Antichain Principle}, $CAC$ (Form 217 in \cite{HR1998})}: Every infinite poset has an infinite chain or an infinite antichain.
\item {\em There are no amorphous sets} \textbf{(Form 64 in \cite{HR1998})}.
\item {\em CS} (see \cite{HST2016}): Every poset without a maximal element has two disjoint cofinal subsets.
\item {\em CWF} (see \cite{Tac2018}): Every poset has a cofinal well-founded subset.

\item \textbf{A weaker form of \L{}o\'{s}'s lemma, $LT$ (Form 253 in \cite{HR1998})}: If $\mathcal{A}=\langle A, \mathcal{R}^{\mathcal{A}}\rangle$ is a non-trivial relational $\mathcal{L}$-structure over some language $\mathcal{L}$, and $\mathcal{U}$ be an ultrafilter on a non-empty set $I$, then the ultrapower $\mathcal{A}^{I}/\mathcal{U}$ and $\mathcal{A}$ are elementarily equivalent.  
\end{enumerate}
\end{defn}
%%%%%%%%%%%%%%%%%%%%%
\subsection{Group-theoretical facts} A group $\mathcal{G}$ {\em acts on a set $X$} if for each $g \in \mathcal{G}$ there is a mapping $x \rightarrow gx$ of $X$ into
itself, such that $1x = x$ for every $x \in X$ and $h(gx)=(hg)x$ for every $g,h \in \mathcal{G}$. Alternatively, actions of a group $\mathcal{G}$ on a set $X$ are the same as group homomorphisms from $\mathcal{G}$ to $Sym(X)$. Suppose that a group $\mathcal{G}$ acts on a set $X$. Let $Orb_{\mathcal{G}}(x)=\{gx:g\in \mathcal{G}\}$ be the orbit of $x \in X$ under the action of $\mathcal{G}$, and $Stab_{\mathcal{G}}(x)=\{g\in \mathcal{G}: gx=x\}$ be the stabilizer of $x$ under the action of $\mathcal{G}$. The {\em Orbit-Stabilizer theorem} states that the size of the orbit is the index of the stabilizer, that is $\vert Orb_{\mathcal{G}}(x)\vert = [\mathcal{G} : Stab_{\mathcal{G}}(x)]$. We also recall that different orbits of the action are disjoint and form a partition of $X$ i.e., $X=\bigcup \{Orb_{\mathcal{G}}(x): x\in X\}$.
An {\em alternating group} is the group of even permutations of a finite set. Let $\{G_{i}:i\in I\}$ be an indexed collection of groups. 
Define the following set.
\begin{equation}\prod_{i\in I}^{weak}G_{i}=\left\{f:I\rightarrow \bigcup_{i\in I} G_{i}\;\middle|\; (\forall i\in I)f(i)\in G_{i}, f(i)= 1_{G_{i}} \text{ for all but finitely many } i\right\}. \end{equation} 

The {\em weak direct product} of the groups $\{G_{i}:i\in I\}$ is the set $\prod^{weak}_{i\in I}G_{i}$ with the operation of component-wise multiplicative defined for all $f,g\in \prod^{weak}_{i\in I}G_{i}$ by $(fg)(i)=f(i)g(i)$ for all $i\in I$.
%%%%%%%%%%%%%%%%%%%%%%%%%%%
\subsection{Permutation models.} We start with a ground model $M$ of $ZFA+AC$ where $A$ is a set of atoms. 
Each permutation of $A$ extends uniquely to a permutation of $M$ by $\epsilon$-induction. A permutation model $\mathcal{N}$ of ZFA is determined by a group $\mathcal{G}$ of permutations of $A$ and a normal filter $\mathcal{F}$ of subgroups of $\mathcal{G}$.
Let $\mathcal{G}$ be a group of permutations of $A$ and $\mathcal{F}$ be a normal filter of subgroups of $\mathcal{G}$. For $x\in M$, we denote the symmetric group with respect to $\mathcal {G}$ by $sym_{\mathcal {G}}(x) =\{g\in \mathcal {G} \mid g(x) = x\}$. We say $x$ is {\em $\mathcal{F}$-symmetric} if $sym_{\mathcal{G}}(x)\in\mathcal{F}$ and $x$ is {\em hereditarily $\mathcal{F}$-symmetric} if $x$ and all elements of its transitive closure are
$\mathcal{F}$-symmetric. We define the permutation model $\mathcal{N}$ with respect to $\mathcal{G}$ and $\mathcal{F}$, to be the class of all hereditarily $\mathcal{F}$-symmetric sets. We recall that $\mathcal{N}$ is a model of $ZFA$ (cf. [\cite{Jec1973}, \textbf{Theorem 4.1}]). 
If $\mathcal{I}\subseteq\mathcal{P}(A)$ is a normal ideal, then the filter base $\{$fix$_{\mathcal{G}}E: E\in\mathcal{I}\}$ generates a normal filter over $\mathcal{G}$, where fix$_{\mathcal{G}}E$ denotes the subgroup $\{\phi \in \mathcal{G} : \forall y \in E (\phi(y) = y)\}$ of $\mathcal{G}$. 
Let $\mathcal{I}$ be a normal ideal generating a normal filter $\mathcal{F}_{\mathcal{I}}$ over $\mathcal G$. Let $\mathcal{N}$ be the permutation model determined by $M, \mathcal{G},$ and $\mathcal{F}_{\mathcal{I}}$. We say $E\in \mathcal{I}$ supports a set $\sigma\in \mathcal{N}$ if fix$_{\mathcal{G}}E\subseteq sym_{\mathcal{G}} (\sigma$). 

\begin{lem}
The following hold.
\begin{enumerate}
    \item In every Fraenkel-Mostowski permutation model, $CS$ implies $vDCP$ (cf. [\cite{HST2016}, \textbf{Theorem 3.15(3)}]).
    \item In ZFA, $CWF$ implies $LW$ (cf. [\cite{Tac2018}, \textbf{Lemma 5}]).
    \item In ZFA, $MC$ implies $CS$ (cf. [\cite{HST2016}, \textbf{Theorem 3.12}]) 
\end{enumerate}
\end{lem}

In this paper, 
\begin{itemize}
    \item Fix a natural number $n\geq 2$. We denote by $\mathcal{N}_{HT}^{1}(n)$ the permutation model constructed in [\cite{HT2020}, \textbf{Theorem 8}]. 
    \item We denote by $\mathcal{N}_{1}$ the basic Fraenkel model.
    \item We denote by  $\mathcal{N}_{HT}^{2}$ the permutation model constructed in [\cite{HT2020}, \textbf{Theorem 10(ii)}].
    \item We denote by  $\mathcal{N}_{2}$ the Second Fraenkel model.
    \item Fix a prime $p\in \omega$. We denote by  $\mathcal{N}_{22}(p)$ the permutation model constructed in [\cite{HT2013}, \textbf{$\S$4.4}].
    \item Fix a natural number $n$ such that $n= 3$ or $n >4$ and an infinite well-ordered cardinal number $\kappa$. We denote by  $\mathcal{M}_{\kappa,n}$ the permutation model constructed in \textbf{Theorem 5.3}. 
\end{itemize}

\subsection{Loeb's theorem}A topological space $(X,\tau)$ is called {\em compact} if for every $U \subseteq \tau$ such that
$\bigcup U = X$ there is a finite subset $V\subseteq U$ such that $\bigcup V = X$.
\begin{lem}{\em (cf. [\cite{Loeb1965}, \textbf{Theorem 1}]).}
{\em Let $\{X_{i}\}_ {i \in I}$ be a family of compact spaces which is indexed by a set $I$ on which there is a well-ordering $\leq$. If $I$ is an infinite set and there is a choice function $F$ on the collection $\{C$ : C is
closed, $C \not= \emptyset, C \subset X_{i}$ for some $i\in I\}$, then the product space
$\prod_{i\in I}X_{i}$ is compact in the product topology.}
\end{lem}

\subsection{A theorem of Fulkerson and Gross} 
Fulkerson--Gross \cite{FG1965} proved the following lemma.
\begin{lem}{\em (cf. [\cite{Kom2015}, \textbf{Lemma 1}], \cite{FG1965}).}
A finite graph $(V, X)$ is chordal if and only if there is an ordering $<$ of $V$ such that $\{w < v : \{w, v\} \in X\}$ is a clique for each
$v \in V$.
\end{lem}
%%%%%%%%%%%%%%%%%%%
\section{Graph theoretical observations}
\subsection{Maximal independent set} 

\begin{prop}(ZF)
{\em Every graph based on a well-ordered set of vertices has a maximal independent set.}
\end{prop} 

\begin{proof}
Let $G=(V_{G},E_{G})$ be a graph on a well-ordered set of vertices $V_{G}=\{v_{\alpha}: \alpha<\lambda\}$.
Thus we can use transfinite recursion, without using any form of choice, to construct a maximal independent set. Let $M_{0}=\emptyset$. Clearly, $M_{0}$ is an independent set. For any ordinal $\alpha$, if $M_{\alpha}$ is a maximal independent set, then we are done. Otherwise, there is some $v \in V_{G} \backslash M_{\alpha}$, where $M_{\alpha}\cup \{v\}$ is an independent set of vertices. In that case, let $M_{\alpha+1}= M_{\alpha}\cup\{v\}$. For limit ordinals $\alpha$, we use $M_{\alpha}=\bigcup_{i\in \alpha} M_{i}$. Clearly, $M =\bigcup_{i\in \lambda} M_{i}$ is a maximal independent set.
\end{proof}

\begin{prop}{(ZF)}
For every $n\in \omega\backslash\{0,1\}$, $\mathcal{P}_{n}$ is equivalent to $AC_{n}$.
\end{prop}

\begin{proof}
($\Leftarrow$) Fix $n\in \omega\backslash\{0,1\}$, and let us assume $AC_{n}$. Let $G=(V_{G},E_{G})$ be a graph from the class $P_{K_{n}}$ (cf. $\S$\textbf{1.1}, for definition of $P_{K_{n}}$). Let $\{G_{i}\}_{i\in I}=\{(V_{G_{i}},E_{G_{i}})\}_{i\in I}$ be the components of $G$. By $AC_{n}$ select $g_{i}\in V_{G_{i}}$ for each $i\in I$. We can see that $J=\{g_{i}:i\in I\}$ is a maximal independent set of $G$. For any $g_{i}, g_{j}\in J$ such that $g_{i}\not=g_{j}$, we have $\{g_{i},g_{j}\}\not\in E_{G}$. Consequently, $J$ is an independent set. For the sake of contradiction, suppose $J$ is not a maximal independent set. Then there is an independent set $L$ which must contain two vertices $x$ and $y$ from $V_{G_{i}}$ for some $i\in I$. Since $\{x,y\}\in E_{G}$, we obtain a contradiction.

($\Rightarrow$) Fix $n\in \omega\backslash\{0,1\}$, and let us assume $\mathcal{P}_{n}$. Consider a system of $n$-element sets $\mathcal{A}=\{A_{i}\}_{i\in I}$.
We construct a graph $G = (V_{G},E_{G})$. 

{\bf{\underline{Constructing $G$}:}} Let $V_{G}$ consist of all the pairs $(Y, y)$ such that $Y \in \mathcal{A}$ and $y \in Y$, and the edge set is defined as follows $\{(Y_{1}, y_{1}),(Y_{2}, y_{2})\}\in E_{G}$ if and only if $Y_{1} = Y_{2}$ and $y_{1} \not= y_{2}$.

Clearly, the components of $G$ are $K_{n}$. By $\mathcal{P}_{n}$, $G$ has a maximal independent set $M$. Since $M$ is an independent set, for each $Y \in \mathcal{A}$ there is at most one $y \in Y$ such that $(Y, y) \in M$. Since $M$ is a maximal independent set, there is at least one $y \in Y$ such that $(Y, y) \in M$. Consequently, $M$ determines a choice function for $\mathcal{A}$. 
\end{proof}

\begin{prop}
(ZF) $AC_{fin}^{\omega}$ is equivalent to $\mathcal{P}_{lf,c}$.
\end{prop}

\begin{proof}($\Rightarrow$) We assume $AC_{fin}^{\omega}$. Let $G = (V_{G}, E_{G})$ be some non-empty locally finite, connected graph. Consider some $r \in V_{G}$. Let $V_{0}=\{r\}$. For each integer $n \geq 1$, define $V_{n} = \{v \in V_{G} : d_{G}(r, v) = n\}$ where `$d_{G}(r, v) = n$' means there are $n$ edges in the shortest path joining $r$ and $v$.
Each $V_{n}$ is finite by locally finiteness of $G$, and $V_{G} = \bigcup_{n\in \omega}V_{n}$ by connectedness of $G$. By $UT(\aleph_{0},fin,\aleph_{0})$ (which is equivalent to $AC_{fin}^{\omega}$(cf. \textbf{Definition 2.4})), $V_{G}$ is countable. Consequently, $V_{G}$ is well-ordered. The rest follows from \textbf{Proposition 3.1}.

($\Leftarrow$) We assume $\mathcal{P}_{lf,c}$. Since $AC_{fin}^{\omega}$ is equivalent to its partial version $PAC_{fin}^{\omega}$ (cf. \textbf{Definition 2.4} or \cite{HR1998}), it suffices to show $PAC_{fin}^{\omega}$.
Let $\mathcal{A}=\{A_{n}:n\in \omega\}$ be a denumerable set of non-empty finite sets. Without loss of generality, we assume that $\mathcal{A}$ is disjoint. Consider a denumerable sequence $T=\{t_{n}:n\in \omega\}$ disjoint from $\mathcal{A}$. We construct a graph $G=(V_{G}, E_{G})$. 
\begin{figure}[!ht]
\centering
\begin{minipage}{\textwidth}
\centering
\begin{tikzpicture}
\draw (-2.5, 1) ellipse (2 and 1);
\draw (-4,1) node {$\bullet$};
\draw (-3,1) node {$\bullet$};
\draw (-2,1) node {$\bullet$};
\draw (-1.5,1) node {...};
\draw (-0.5,1.6) node {$A_{1}$};
\draw (-4,1) -- (-3,1);
\draw (-3,1) -- (-2,1);
\draw (-2,1) to[out=-70,in=-70] (-4,1);
\draw (-2.5,-1) node {$\bullet$};
\draw (-2.5,-1.3) node {$t_{1}$};
\draw (-4,1) -- (-2.5,-1);
\draw (-3,1) -- (-2.5,-1);
\draw (-2,1) -- (-2.5,-1);
\draw (2.5, 1) ellipse (2 and 1);
\draw (1,1) node {$\bullet$};
\draw (2,1) node {$\bullet$};
\draw (3,1) node {$\bullet$};
\draw (3.5,1) node {...};
\draw (4.5,1.6) node {$A_{2}$};
\draw (1,1) -- (2,1);
\draw (2,1) -- (3,1);
\draw (1,1) to[out=-70,in=-70] (3,1);
\draw (2.5,-1) node {$\bullet$};
\draw (2.5,-1.3) node {$t_{2}$};
\draw (3,1) -- (2.5,-1);
\draw (2,1) -- (2.5,-1);
\draw (1,1) -- (2.5,-1);
\draw (-2.5,-1) -- (2.5,-1);
\draw (3.5,-1) node {...};
\draw (5,1) node {...};
\end{tikzpicture}
\end{minipage}
\caption{\em The graph $G$.}
\end{figure}
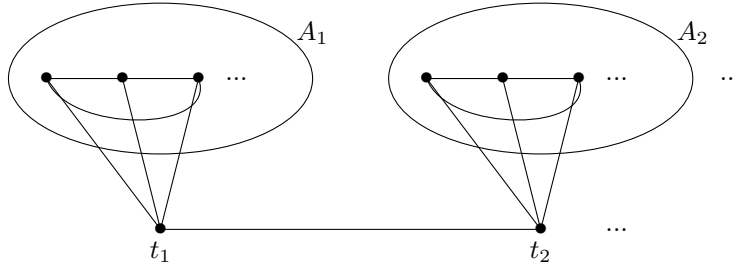

{\bf{\underline{Constructing $G$}:}}
Let $V_{G} = (\bigcup_{n\in \omega}A_{n})\cup T$. For each $n \in\omega$, let $\{t_{n},t_{n+1}\}\in E_{G}$ and $\{t_{n},x\}\in E_{G}$ for every element $x\in A_{n}$. Also for each $n \in\omega$, and any two $x,y\in A_{n}$ such that $x\not=y$, let $\{x,y\}\in E_{G}$ (see Figure 1).

Clearly, the graph $G$ is connected and locally finite. By assumption, $G$ has a maximal independent set of vertices, say $M$. Since $M$ is maximal, $M$ has to be infinite. Moreover, for each $i\in \omega$, either $t_{i}\in M$ or some $v\in A_{i}$ is in $M$.  Since $M$ is an independent set, for each $i \in \omega$ there is at most one $v \in A_{i}$ such that $v \in M$.
Define $M'=\{v\in M: v\in A_{i}$ for some $i\in \omega\}$. If $M'$ is finite, then since $\{t_{n},t_{n+1}\} \in E_{G}$ for all $n \in \omega$, it follows that for some $n \in \omega$, $M \cap (A_{n}\cup \{t_{n}\}) = \emptyset$. Then for any $u \in A_{n}$, $M \cup \{u\}$ is an independent set
which properly contains $M$, contradicting $M$’s being a maximal independent set. Thus $M'$ is infinite, which clearly yields a partial choice function for $\mathcal{A}$.
%%%%%%%%%%%%%%%%%%%%%%%%%%%%
\end{proof}
%%%%%%%%%%%%%%%%%%%%%%%%%%%%%%%%
\begin{prop}(ZF)
{\em $UT(\aleph_{0},\aleph_{0},\aleph_{0})$ implies $\mathcal{P}_{lc,c}$, and $\mathcal{P}_{lc,c}$ implies $AC_{\aleph_{0}}^{\aleph_{0}}$.}
\end{prop}
\begin{proof}
In order to prove the first implication, let $G = (V_{G}, E_{G})$ be some non-empty locally countable connected graph. Consider some $r \in V_{G}$. Let $V_{0}=\{r\}$. For each integer $n \geq 1$, define $V_{n} = \{v \in V_{G} : d_{G}(r, v) = n\}$. Since $G$ is locally countable, each $V_{n}$ is countable by $UT(\aleph_{0},\aleph_{0},\aleph_{0})$. Also $V_{G} = \bigcup_{n\in \omega}V_{n}$ since $G$ is connected. By $UT(\aleph_{0},\aleph_{0},\aleph_{0})$, $V_{G}$ is countable. The rest follows from \textbf{Proposition 3.1}. The second assertion follows from the arguments of \textbf{Proposition 3.3}, since $AC_{\aleph_{0}}^{\aleph_{0}}$ is equivalent to $PAC_{\aleph_{0}}^{\aleph_{0}}$ in ZF (cf. \textbf{Definition 2.4} or \cite{HR1998}).
\end{proof}

\begin{Remark}
Fix $n\in \omega\backslash\{0,1\}$. We denote by $P_{C_{n}}$, the class of those graphs whose only components are $C_{n}$.
We denote by $\mathcal{P}'_{n}$ the statement {\em `Every graph from the class $P_{C_{n}}$, has a maximal independent set'}. 
We remark that $AC_{P_{n}}$ implies $\mathcal{P}'_{n}$ in ZF where ${P}_{n}$ is the {\em Perrin number} of $n$ (Perrin numbers are defined by the recurrence relation $P(n) = P(n-2) + P(n-3)$ for $n > 2$, where the initial values are $P(0) = 3, P(1) = 0$, and $P(2) = 2$).
Let $G=(V_{G},E_{G})$ be a graph from the class $P_{C_{n}}$. Let $\{G_{i}\}_{i\in I}=\{(V_{G_{i}},E_{G_{i}})\}_{i\in I}$ be the components of $P_{C_{n}}$. Let $M_{i}$ be the collection of different maximal independent sets of $G_{i}$ for each $i\in I$. Since the number of different maximal independent sets in each component is $P_{n}$\footnote{We use the fact that the number of different maximal independent sets in an $n$-vertex cycle graph is the $n$-th Perrin number for $1<n <\omega$ (cf. \cite{Fur1987}).}, by $AC_{P_{n}}$ we can choose a $m_{i}\in M_{i}$ for each $i\in I$. Clearly, $\bigcup_{i\in I} m_{i}$ is a maximal independent set of $G$. 
\end{Remark} 
%%%%%%%%%%%%%%%%%%%
\subsection{The graph homomorphism problem} 
\begin{prop}{(ZF)} {\em $\mathcal{P}_{G,H_{2}}$ is equivalent to $AC_{2}$.}
\end{prop}

\begin{proof}
As mentioned in subsection 1.4, for any $n \in \omega\backslash \{0, 1\}$, $\mathcal{P}_{G,H_{n}}$ implies $AC_{n}$. We prove that $AC_{2}$ implies $\mathcal{P}_{G,H_{2}}$ in ZF. Let $H=(V_{H}, E_{H})$ be a graph such that $V_{H}=\{v_{1},v_{2}\}$, and $G=(V_{G}, E_{G})$ be an infinite graph. We assume that every finite subgraph of $G$ has a homomorphism into $H$. Let $I$ be the set of components of $G$. 

\textbf{Case (1). $\{v_{1},v_{1}\}\in E_{H}$ or $\{v_{2},v_{2}\}\in E_{H}$.} If $\{v_{1},v_{1}\}\in E_{H}$, then for any $G'=(V_{G'}, E_{G'})\in I$, $f_{G'}:G'\rightarrow H$ defined by $f_{G'}(x)=v_{1}$ for every $x\in V_{G'}$, is a homomorphism from $G'$ to $H$. The function $f:G\rightarrow H$, defined by $f(x)=f_{G'}(x)$ for $G'=(V_{G'}, E_{G'})\in I$ and $x\in V_{G'}$, is a homomorphism from $G$ to $H$. The case $\{v_{2},v_{2}\}\in E_{H}$ is similar.

\textbf{Case (2). $E_{H}=\{\{v_{1},v_{2}\}\}$.} 
We follow the proof of $AC_{2}$ implies the $2$-coloring problem in ZF (cf. \cite{Myc1964}). Fix an arbitrary $G'=(V_{G'}, E_{G'})\in I$ and select an arbitrary element $a\in V_{G'}$. The function $f_{G'}:G'\rightarrow H$ defined by $f_{G'}(v)=v_{1}$ if there is an odd number of vertices between $a$ and $v$ in the shortest path from $a$ to $v$ and $f_{G'}(v)=v_{2}$ otherwise, is a homomorphism from $G'$ to $H$.
Clearly, the set of homomorphisms $\phi:G'\rightarrow H$ contains precisely two elements. By $AC_{2}$, there exists a family $\{f_{G'}\}_{G'\in I}$ of homomorphisms $f_{G'}:G'\rightarrow H$. The function $f:G\rightarrow H$, defined by $f(x)=f_{G'}(x)$ for $G'=(V_{G'}, E_{G'})\in I$ and $x\in V_{G'}$ is a homomorphism from $G$ to $H$. 

\textbf{Case (3). $E_{H}=\emptyset$.} Then $G$ must be a discrete graph with no edges (by the assumption that every finite subgraph of $G$ has a homomorphism into $H$) and any possible mapping of vertices from $V_{G}$ to either $v_{1}$ or $v_{2}$ gives a homomorphism. 
Define a function $f:G\rightarrow H$ such that $x \mapsto v_{1}$ for each $x\in V_{G}$.
Clearly, $f$ is a homomorphism from $G$ into $H$ without using any form of choice.
\end{proof}
%%%%%%%%%%%%%%%%%%%%%%%%
\subsection{Locally finite connected graphs} 
\begin{prop}
(ZF) {\em $AC_{fin}^{\omega}$ implies $\mathcal{P}_{G,H}$, if $G$ is locally finite and connected.}
\end{prop}

\begin{proof} Let $G = (V_{G}, E_{G})$ be some non-empty locally finite, connected graph. Consider some $r \in V_{G}$. Let $V_{0}=\{r\}$. For each integer $n \geq 1$, define $V_{n} = \{v \in V_{G} : d_{G}(r, v) = n\}$. Each $V_{n}$ is finite by locally finiteness of $G$, and $V_{G} = \bigcup_{n\in \omega}V_{n}$ by connectedness of $G$. By $AC_{fin}^{\omega}$, $V_{G}$ is countable. We know that $\mathcal{P}_{G,H}$ holds in ZF, if $G$ is based on a well-ordered set of vertices (cf.\cite{BG2020}).
\end{proof}

\begin{prop} (ZF)$AC_{fin}^{\omega}$ implies the statement {\em `If $(V,X)$ is a connected locally finite chordal graph, then there is an ordering $<$ of $V$ such that $\{w < v : \{w,v\} \in X\}$ is a clique for each $v\in V$'.}
\end{prop}

\begin{proof} We note that by arguments in the proof of \textbf{Proposition 3.7}, it is enough to see that the statement {\em `If $(V,X)$ is a chordal graph based on a well orderable set of vertices, then there is an ordering $<$ of $V$ such that $\{w < v : \{w,v\} \in X\}$ is a clique for each $v\in V$'} is provable in ZF.
By \textbf{Lemma 2.7}, each finite subgraph $(W, X|W)$ has an ordering such
that $\{w < v : \{w, v\} \in X\restriction W\}$ is a clique for every $v \in W$. We can encode every
total ordering of a set $W$ by a choice of one of $<, =, >$ for each pair $(x, y) \in W\times W$. Endow $\{<, =, >\}$ with the discrete topology and $T = \{<, =, >
\}^{V\times V}$ with the product topology. Since $V$ is well-ordered, $V\times V$ is well-ordered in ZF. Consequently, $\{<,=,>\}\times \{V\times V\}$ is well-ordered in ZF. By \textbf{Lemma 2.6}, $T$ is compact. We use the compactness of $T$ to prove the existence of the desired ordering.
\end{proof}

\begin{Remark} We list some other graph-theoretical statements from different papers, restricted to locally finite connected graphs, which are related to $AC_{fin}^{\omega}$.
\begin{enumerate}
\item Komj\'{a}th--Galvin \cite{KG1991} proved that any graph based on a well-ordered set of vertices has a chromatic number and an irreducible good coloring in ZF. Consequently, the statements {\em `any locally finite connected graph has a chromatic number'} and {\em `any locally finite connected graph has an irreducible good coloring'} are provable under $AC_{fin}^{\omega}$ in ZF.

\item Hajnal [\cite{Haj1985}, \textbf{Theorem 2}] proved that if the chromatic number of a graph $G_{1}$ is finite (say $k<\omega$), and the chromatic number of another graph $G_{2}$ is infinite, then the chromatic number of $G_{1}\times G_{2}$ is $k$. In \cite{BG2020} we observed that if $G_{1}$ is based on a well-ordered set of vertices, then the following statement holds in ZF.
\begin{center}
    {\em `$\chi(E_{G_{1}})=k<\omega$ and $\chi(E_{G_{2}})\geq\omega$ implies $\chi(E_{G_{1}\times G_{2}})=k$.'}
\end{center}
Consequently, under $AC_{fin}^{\omega}$ the above statement holds in ZF if $G_{1}$ is a locally finite connected graph. 
 
\item Delhomm\'{e} and Morillon \cite{DM2006} proved that $AC_{fin}^{\omega}$ is equivalent to the statement {\em `Every locally finite connected graph has a spanning tree'} in ZF.
\item The $n$-coloring theorem restricted to locally finite connected graphs is provable under $AC_{fin}^{\omega}$ in ZF by \textbf{Proposition 3.7}.
\end{enumerate}
\end{Remark}

%%%%%%%%%%%%%%%%%%%%%%%%%%%
\section{A variant of CAC}
Tachtsis communicated to us the following lemma.

\begin{lem} The following holds.
\begin{enumerate}
    \item $UT(\aleph_{0},\aleph_{0},\aleph_{0})$ implies the statement {\em `If $(P, \leq)$ is a poset such that P is well-ordered, and if all antichains in P are finite and all chains in P are countable, then P is countable'}.
    \item “$\aleph_{1}$ is regular” implies the statement {\em `If $(P, \leq)$ is a poset such that P is well-ordered, and if all antichains in P are finite and all chains in P are countable, then P is countable'}.
\end{enumerate}
\end{lem}

\begin{proof}
We prove (1). Let $(P,\leq)$ be a poset such that $P$ is well-ordered, all antichains in $P$ are finite, and all chains are countable. Fix a well-ordering $\preceq$ of $P$. By way of contradiction, assume that $P$ is uncountable.\footnote{Since we study in set theory without choice, we note that a set $X$ is uncountable if $\vert X\vert \not\leq \aleph_{0}$. We also note that without choice, “uncountable” may not generally have a clear meaning;
for example, another definition could be that $X$ is uncountable if $\aleph_{0} < \vert X\vert$ (meaning that there is an injection from $\omega$ into $X$ but not vice versa). The above two definitions are clearly equivalent in ZFC, but they are not equivalent in ZF.} 
We construct an infinite antichain to obtain a contradiction. 
Since $P$ is well-ordered by $\preceq$, we may construct (via transfinite induction) a maximal $\leq$-chain, $V_0$ say, without invoking any form of choice. Since $V_0$ is countable, it follows that $P-V_{0}$ is uncountable and every element of $P-V_{0}$ is incomparable to some element of $V_0$. Thus $P-V_{0} = \bigcup\{W_{p} : p\in V_{0}\}$, where $W_{p}$ is the set of all elements of $P-V_{0}$ which are incomparable to $p$. Since $P-V_{0}$ is uncountable and $V_0$ is countable, it follows by $UT(\aleph_{0},\aleph_{0},\aleph_{0})$ that $W_{p}$ is uncountable for some $p$ in $V_{0}$. Let $p_0$ be the least (with respect to $\preceq$) such element of $V_0$.
Now, construct a maximal $\leq$-chain in (the uncountable set) $W_{p_{0}}$, $V_1$ say, and let (similarly to the above argument) $p_1$ be the least (with respect to $\preceq$) element of $V_1$ such that the set $W_{p_{1}}$ of all elements of $W_{p_{0}}$ which are incomparable to $p_{1}$ is uncountable.
Continuing in this fashion by induction (and noting that the process cannot stop at a finite stage), we obtain a countably infinite antichain $\{p_{n}: n\in\omega\}$, contradicting the assumption that all antichains are finite. Therefore, $P$ is countable.

Similarly, we can prove (2).
\end{proof}

Modifying \textbf{Lemma 4.1}, we may observe that $UT(\aleph_{\alpha},\aleph_{\alpha},\aleph_{\alpha})$ implies the statement {\em `If $(P, \leq)$ is a poset such that P is well-ordered, and if all antichains in P are finite and all chains in P have size $\aleph_{\alpha}$, then $P$ has size $\aleph_{\alpha}$'} for any regular $\aleph_{\alpha}$ in ZF. 

\begin{corr}
The statement {\em `If $(P, \leq)$ is a poset such that P is well-ordered, and if all antichains in P are finite and all chains in P are countable, then P is countable'} holds in any Fraenkel-Mostowski model.
\end{corr}
\begin{proof}
Follows from the fact that the statement “$\aleph_{1}$ is a regular cardinal” holds in every Fraenkel-Mostowski model (cf. [\cite{HKRST2001}, \textbf{Corollary 1}]).
\end{proof}

\begin{thm}(ZFA) {\em Let $n \in \omega \backslash \{0, 1\}$. The statement “For every regular $\aleph_{\alpha}$, $CAC^{\aleph_{\alpha}}_{1}$” implies neither $AC_{n}^{-}$ nor “there are no amorphous sets”.}
\end{thm}

\begin{proof}
Halbeisen--Tachtsis [\cite{HT2020}, \textbf{Theorem 8}] constructed a permutation model
(we denote by $\mathcal{N}_{HT}^{1}(n)$) where for arbitrary $n\geq 2$, $AC_{n}^{-}$ fails but CAC holds. We fix an arbitrary integer $n\geq 2$ and recall the model constructed in the proof of [\cite{HT2020}, \textbf{Theorem 8}] as follows.

{\bf{\underline{Defining the ground model $M$:}}}
We start with a ground model $M$ of $ZFA+AC$ where $A$ is a countably infinite set of atoms written as a disjoint union $\bigcup\{A_{i}:i\in \omega\}$ where for each $i\in \omega$, $A_{i}=\{a_{i_{1}},a_{i_{2}},..., a_{i_{n}}\}$ and $\vert A_{i}\vert = n$. 

{\bf{\underline{Defining the group $\mathcal{G}$ and the filter $\mathcal{F}$ of subgroups of $\mathcal{G}$:}}}
\begin{itemize}
    \item {\bf{\underline{Defining $\mathcal{G}$:}}} $\mathcal{G}$ is defined in \cite{HT2020} in a way so that \textbf{if $\eta\in \mathcal{G}$, then $\eta$ only moves finitely many atoms} and for all $i\in \omega$, $\eta(A_{i})=A_{k}$ for some $k\in \omega$. We recall the details from \cite{HT2020} as follows. For all $i\in \omega$, let $\tau_{i}$ be the $n$-cycle $a_{i_{1}}\mapsto a_{i_{2}}\mapsto ... \mapsto a_{i_{n}}\mapsto a_{i_{1}}$. For every permutation $\psi$ of $\omega$, which moves only finitely many natural numbers, let $\phi_{\psi}$ be the permutation of $A$ defined by $\phi_{\psi}(a_{i_{j}})=a_{\psi(i)_{j}}$ for all $i\in \omega$ and $j=1,2,...,n$. Let $\eta\in \mathcal{G}$ if and only if $\eta=\rho \phi_{\psi}$ where $\psi$ is a permutation of $\omega$ which moves only finitely many natural numbers and $\rho$ is a permutation of $A$ for which there is a finite $F\subseteq\omega$ such that for every $k\in F$, $\rho\restriction A_{k}=\tau^{j}_{k}$ for some $j<n$, and $\rho$ fixes $A_{m}$ pointwise for every $m\in \omega\backslash F$. 
    \item {\bf{\underline{Defining $\mathcal{F}$:}}} Let $\mathcal{F}$ be the filter of subgroups of $\mathcal{G}$ generated by $\{$fix$_{\mathcal{G}}(E): E\in [A]^{<\omega}\}$. 
    \end{itemize}
{\bf{\underline{Defining the permutation model:}}}
Consider the FM-model $\mathcal{N}_{HT}^{1}(n)$ determined by $M$, $\mathcal{G}$ and $\mathcal{F}$.

Following \textbf{point 1} in the proof of [\cite{HT2020}, \textbf{Theorem 8}], both $A$ and $\mathcal{A}=\{A_{i}\}_{i\in \omega}$ are amorphous in $\mathcal{N}_{HT}^{1}(n)$ and no infinite subfamily $\mathcal{B}$ of $\mathcal{A}$ has a Kinna--Wagner selection function. Consequently, $AC^{-}_{n}$ fails. 
We prove that for any regular $\aleph_{\alpha}$, $CAC^{\aleph_{\alpha}}_{1}$ holds in $\mathcal{N}_{HT}^{1}(n)$. Let $(P,\leq)$ be a poset in $\mathcal{N}_{HT}^{1}(n)$ such that all antichains in $P$ are finite and all chains in $P$ have size $\aleph_{\alpha}$. Let $E\in [A]^{<\omega}$ be a support of $(P,\leq)$. Following the arguments of \textbf{[\cite{Tac2016}, claim 3]} we can see that for each $p \in P$, the set $Orb_{E}(p)=\{\phi(p): \phi\in$ fix$_{\mathcal{G}}(E)\}$ is an anti-chain in $P$. Following the arguments of \textbf{[\cite{Tac2016}, claim 4]} we can see that $P$ can be expressed as a well-orderable union of antichains. In fact, $\mathcal{O}=\{Orb_{E}(p): p\in P\}$ is a well-ordered partition of $P$. We note that all antichains in $P$ are finite, and hence well-orderable. Consequently, $P$ is well-orderable in $\mathcal{N}_{HT}^{1}(n)$ since $UT(WO,WO,WO)$ holds in $\mathcal{N}_{HT}^{1}(n)$. We also note that $UT(WO,WO,WO)$ implies UT($\aleph_{\alpha}$,$\aleph_{\alpha}$,$\aleph_{\alpha}$) in any FM-model (cf. page 176 of \cite{HR1998}).  So, we are done by  \textbf{Lemma 4.1} and the point noted in the paragraph after \textbf{Lemma 4.1} (cf. the arguments of \textbf{[\cite{Tac2016}, claim 5]} as well).
\end{proof}

\begin{Remark} We can see that in the basic Fraenkel model (labeled as Model $\mathcal{N}_{1}$ in \cite{HR1998}) the statement “For every regular $\aleph_{\alpha}$, $CAC^{\aleph_{\alpha}}_{1}$” holds. We recall that $UT(WO,WO,WO)$ holds in $\mathcal{N}_{1}$ (cf. \cite{HR1998}). Fix a regular $\aleph_{\alpha}$. Let $(P,\leq)$ be a poset in $\mathcal{N}_{1}$, and $E$ be a ﬁnite support of $(P,\leq)$. By the arguments of the proof of \textbf{Theorem 4.3}, $\mathcal{O}=\{Orb_{E}(p): p\in P\}$ is a well-ordered partition of $P$. Now for each $p\in P$, $Orb_{E}(p)$ is an antichain (cf. the proof of [\cite{Jec1973}, \textbf{Lemma 9.3}]). Thus, by methods from the proof of \textbf{Theorem 4.3}, $CAC^{\aleph_{\alpha}}_{1}$ holds in $\mathcal{N}_{1}$. In $\mathcal{N}_{1}$, the statement `there are no amorphous sets' is false. For reader's information we note that $AC_{n}^{-}$ and `there are no amorphous sets' are independent of each other.
\end{Remark}

\begin{thm}{\em (ZF) $CAC^{\aleph_{0}}_{1}$ implies $PAC^{\aleph_{1}}_{fin}$.}
\end{thm}
\begin{proof}
Let $\mathcal{A}=\{A_{n} : n \in \aleph_{1}\}$ be a family of non-empty ﬁnite sets. Without loss of generality, we assume that $\mathcal{A}$ is disjoint. Deﬁne a binary relation $\leq$ on $A =\bigcup \mathcal{A}$ as follows: for all $a,b \in A$, let $a\leq b$ if and only if $a = b$ or $a \in A_{n}$ and $b \in A_{m}$ and $n < m$. Clearly, $\leq$ is a partial order on $A$. Also, $A$ is uncountable. The only antichains of $(A,\leq)$ are the ﬁnite sets $A_{n}$ and subsets of $A_{n}$ where $n\in \aleph_{1}$. By $CAC^{\aleph_{0}}_{1}$, $A$ has an uncountable chain, say $C$. Let $M =\{m \in \aleph_{1} : C \cap A_{m}\not= \emptyset\}$. Since $C$ is a chain and $\mathcal{A}$ is the family of all antichains of $(A,\leq)$, we have $M =\{m \in \aleph_{1} : \vert C \cap A_{m}\vert =1\}$. Clearly, $f =\{ (m,c_{m}) : m \in M\}$, where for $m \in M$, $c_{m}$ is the unique element of $C \cap A_{m}$, is a choice function of the uncountable subset $\mathcal{B}=\{A_{m} : m \in M\}$ of $\mathcal{A}$. Thus $\mathcal{B}$ is an $\aleph_{1}$-sized subfamily of $\mathcal{A}$ with a choice function.
\end{proof}

\begin{corr}
{\em There exists a model of
ZF in which $DC$ holds and $PAC^{\aleph_{1}}_{fin}$ fails, and thus $CAC_{1}^{\aleph_{0}}$ also fails.}
\end{corr}

\begin{proof}
 We refer the reader to Jech [\cite{Jec1973}, \textbf{Theorem 8.3}] by noting that $\aleph_{\alpha}$ therein can be replaced by $\aleph_{1}$. We also note that the fact that $PAC^{\aleph_{1}}_{fin}$ is false in the model follows immediately from [\cite{Jec1973}, \textbf{Theorem 8.3(iii)}]. The rest follows from \textbf{Theorem 4.5}.
\end{proof}
%%%%%%%%%%%%%%%%%%%%%%%%%%%%%%
\section{Cofinal well-founded subsets in ZFA}

Tachtsis [\cite{Tac2018}, \textbf{Theorem 10(ii)}] proved that CWF holds in the basic Fraenkel model.
Howard, Saveliev, and Tachtsis [\cite{HST2016}, \textbf{Theorem 3.26}] proved that CS holds in the basic Fraenkel model. We modify the arguments from [\cite{HST2016},\textbf{Theorem 3.26}] and  [\cite{Tac2018}, \textbf{Theorem 10(ii)}] to observe the following.

\begin{lem}
{\em Let $A$ be a set of atoms. Let $\mathcal{G}$ be the group of permutations of A such that either each $\eta\in \mathcal{G}$ moves only finitely many atoms or there is a $n\in\omega\backslash \{0,1\}$, such that for all $\eta\in \mathcal{G}$, $\eta^{n}=1_{A}$.
Let $\mathcal{F}$ be the normal filter of subgroups of $\mathcal{G}$ generated by $\{$fix$_{\mathcal{G}}(E): E\in [A]^{<\omega}\}$. Then in the Fraenkel-Mostowski model $\mathcal{N}$ determined by $A$, $\mathcal{G}$, and $\mathcal{F}$, CS and $CWF$ hold. Consequently, vDCP and LW hold.
}
\end{lem}

\begin{proof} We follow the steps below.
\begin{enumerate}
\item Let $(P,\leq)$ be a poset in $\mathcal{N}$ and $E\in [A]^{<\omega}$ be a support of $(P,\leq)$.
We can write $P$ as a disjoint union of ﬁx$_{\mathcal{G}}(E)$-orbits, i.e., $P=\bigcup \{Orb_{E}(p):p
\in P\}$, where $Orb_{E}(p)=\{\phi(p): \phi\in$ ﬁx$_{\mathcal{G}}(E)\}$ for all $p \in P$. The family $\{Orb_{E}(p): p \in P\}$ is well-orderable in $\mathcal{N}$ since ﬁx$_{\mathcal{G}}(E) \subseteq Sym_{\mathcal{G}}(Orb_{E}(p))$ for all $p \in P$ (cf. the arguments of \textbf{[\cite{Tac2016}, claim 4]}).

\item  We prove that $Orb_{E}(p)$ is an antichain in $P$ for each $p\in P$. Otherwise there is a $p\in P$, such that $Orb_{E}(p)$ is not an antichain in $(P,\leq)$. Thus, for some $\phi,\psi\in$ fix$_{\mathcal{G}}(E)$, $\phi(p)$ and $\psi(p)$ are comparable. Without loss of generality we may assume $\phi(p)<\psi(p)$. Let $\pi=\psi^{-1}\phi$. Consequently, $\pi(p)<p$. 

\textbf{Case 1:} Suppose there is a $n\in\omega\backslash \{0,1\}$, such that for every $\eta\in \mathcal{G}$, $\eta^{n}=1_{A}$. So $\pi^{n}=1_{A}$. Thus, $p=\pi^{n}(p)<\pi^{n-1}(p)<...<\pi(p)<p$. By transitivity of $<$, $p<p$, which is a contradiction.

\textbf{Case 2:} Suppose each $\eta\in \mathcal{G}$, moves only finitely many atoms. Then for some $k<\omega$, $\pi^{k}=1$. Rest follows from the arguments in \textbf{Case 1}.

\item  We can follow [\cite{HST2016}, \textbf{Theorem 3.26}] to see that $CS$ holds in $\mathcal{N}$. 

\item Although in every Fraenkel-Mostowski model,  $CS$ implies $vDCP$ in ZFA (cf. \textbf{Lemma 2.5}), we can recall the arguments from the $1^{st}$-paragraph of [\cite{HST2016}, \textbf{Page175}] to give a direct proof of $vDCP$ in $\mathcal{N}$ without invoking $CS$.

\item We can follow [\cite{Tac2018}, \textbf{Theorem 10 (ii)}] to see that $CWF$ holds in $\mathcal{N}$. 
\item Although $CWF$ implies $LW$ in ZFA (cf. \textbf{Lemma 2.5}), we can recall the arguments from the proof of [\cite{HT2020}, \textbf{Theorem 10(ii)}] to give a direct proof of $LW$ in $\mathcal{N}$ without invoking $CWF$. In particular, using a given linear order in $\mathcal{N}$, the fact that
an element $x$ of $\mathcal{N}$ is well-orderable in $\mathcal{N}$ if fix$_{\mathcal{G}}(x) \in \mathcal{F}$ and a similar argument as in \textbf{Case 1} of step (2), one can verify that $LW$ is true in $\mathcal{N}$ without invoking $CWF$.
\end{enumerate}
\end{proof}
%%%%%%%%%%%%%%%%%%%%%%%%%%%%%
\begin{remark}
The authors of \cite{HST2016} communicated to the referee that in an unpublished manuscript \cite{HST} of theirs, they have shown that $MC$ implies $CWF$ in ZFA. The referee communicated to us the argument with their kind permission. We quote their argument from \cite{HST} for reader's convenience:
``Assume that $MC$ is true. Let $(P,\leq)$ be a (non-empty) poset and also let $F$ be a multiple choice function for $\mathcal{P}(P) \backslash \{\emptyset\}$. Using $F$, a cofinal well-founded subset of $P$ can be recursively defined as follows: Let $A_{0} = P$ and $B_{0} = F(A_{0})$. Assume that for some ordinal $\alpha > 0$, sets $A_{\beta}$ and $B_{\beta}$ are defined for all $\beta < \alpha$. Define $A_{\alpha} = \{p \in P : \forall \beta < \alpha \forall q \in B_{\beta} (p \nleq q)\}$ and $B_{\alpha} = F(A_{\alpha})$, if $A_{\alpha}$ is non-empty. Since On (the class of all ordinal numbers) is a proper class, there is $\gamma \in On$ such that $A_{\gamma} = \emptyset$. Clearly, $B =\bigcup\{B_{\alpha} : \alpha < \gamma\}$ is a cofinal well-founded subset of $P$."
\end{remark}
%%%%%%%%%%%%%%%%%%%%%%%%%
\subsection{A model of ZFA}
Herrlich, Howard, and, Tachtsis [\cite{HHT2012}, \textbf{Theorem 11}, \textbf{Case 1}, \textbf{Case 2}] constructed two different classes of permutation models. Halbeisen--Tachtsis  [\cite{HT2020}, \textbf{Theorem 10(ii)}] proved that LOC$_{2}^{-}$ does not imply $LOKW_{4}^{-}$ in ZFA. For the sake of convenience, we denote by  $\mathcal{N}_{HT}^{2}$, the permutation model of [\cite{HT2020}, \textbf{Theorem 10(ii)}]. The model $\mathcal{N}_{HT}^{2}$ is very similar to the model from [\cite{HHT2012},\textbf{Theorem 11}, \textbf{Case 2}] except the fact that in $\mathcal{N}_{HT}^{2}$ each permutation $\phi$ in the group $\mathcal{G}$ of permutations of the sets of atoms, can move only finitely many atoms. Fix a natural number $n$ such that $n= 3$ or $n >4$ and an infinite well-ordered cardinal number $\kappa$. We construct a model $\mathcal{M}_{\kappa,n}$ of ZFA similar to the model constructed in [\cite{HHT2012},\textbf{Theorem 11}, \textbf{Case 1}], where each permutation $\phi$ in the group $\mathcal{G}$ of permutations of the sets of atoms, can move only finitely many atoms. 

\begin{thm}
{\em Let $n$ be a natural number such that $n= 3$ or $n >4$ and $\kappa$ be an infinite well-ordered cardinal number. Then there is a model $\mathcal{M}_{\kappa,n}$ of ZFA where the following hold.
\begin{enumerate}
    \item If $X\in \{LOC^{-}_{2}, MC\}$, then $X$ holds.
    \item $LOC_{n}^{-}$ fails. 
    \item If $X\in\{\mathcal{P}_{n}$, $\mathcal{P}_{G,H_{n}}, DT, LT\}$, then $X$ fails.
\end{enumerate}
}
\end{thm} 
\begin{proof}
Fix a natural number $n$ such that $n=3$ or $n > 4$ and an infinite well-ordered cardinal number $\kappa$. 

{\bf{\underline{Defining the ground model $M$:}}} We start with a ground model $M$ of $ZFA+AC$ where $A$ is a
$\kappa$-sized set of atoms written as a disjoint union $\bigcup\{A_{\alpha}:\alpha< \kappa\}$, where $A_{\alpha}=\{a_{\alpha,1},a_{\alpha,2},...,a_{\alpha,n}\}$ such that $\vert A_{\alpha}\vert=n$ for all $\alpha<\kappa$.

{\bf{\underline{Defining the group $\mathcal{G}$ and the filter $\mathcal{F}$ of subgroups of $\mathcal{G}$:}}}
\begin{itemize}
    \item {\bf{\underline{Defining $\mathcal{G}$:}}}
    Let $\mathcal{G}$ be the {\em weak direct product} of $\mathcal{G}_{\alpha}$'s where $\mathcal{G}_{\alpha}$ is the alternating group on $\mathcal{A}_{\alpha}$ for each $\alpha<\kappa$.  Hence, a permutation $\eta$ of $A$ is an element of $\mathcal{G}$ if and only if for every $\alpha<\kappa$, $\eta\restriction A_{\alpha} \in \mathcal{G}_{\alpha}$, and $\eta \restriction A_{\alpha} = 1_{A_{\alpha}}$ for all but ﬁnitely many ordinals $\alpha<\kappa$. Consequently, {\em every element $\eta\in \mathcal{G}$ moves only ﬁnitely many atoms}. 
    
    \item {\bf{\underline{Defining $\mathcal{F}$:}}} Let $\mathcal{F}$ be the normal filter of subgroups of $\mathcal{G}$ generated by $\{$fix$_{\mathcal{G}}(E): E\in [A]^{<\omega}\}$. 
    \end{itemize}
{\bf{\underline{Defining the permutation model:}}} Consider the permutation model $\mathcal{M}_{\kappa,n}$ determined by $M$, $\mathcal{G}$ and $\mathcal{F}$.

{\bf{\underline{(1). If $X\in \{LOC^{-}_{2}, MC\}$, then $X$ holds in $\mathcal{M}_{\kappa,n}$:}}}
We note that $MC$ is true in the model $\mathcal{M}_{\kappa,n}$. The proof is fairly similar to the one that $MC$ is true in the Second Fraenkel Model (see \cite{Jec1973}). Applying the group-theoretic facts from [\cite{HHT2012}, \textbf{Theorem 11}, \textbf{Case 1}] and following the arguments of the proof of [\cite{HT2020}, \textbf{Theorem 10(ii)}] we may observe that $LOC_{2}^{-}$ holds in $\mathcal{M}_{\kappa,n}$.

{\bf{\underline{(2). $LOC_{n}^{-}$ fails in $\mathcal{M}_{\kappa,n}$:}}} We prove that in $\mathcal{M}_{\kappa,n}$, the well-ordered family $\mathcal{A} = \{A_{\alpha} : \alpha< \kappa\}$ of $n$-element sets does not have a partial choice function. For the sake of contradiction, let $\mathcal{B}$ be an inﬁnite subfamily of $\mathcal{A}$ with a choice function $f \in \mathcal{M}_{n}$ and support $E\in [A]^{<\omega}$. Since $E$ is finite, there is an $i<\kappa$ such that $A_{i} \in \mathcal{B}$ and $A_{i}\cap E = \emptyset$. Without loss of generality, let $f(A_{i})=a_{i_{1}}$. Consider the permutation $\pi$ which is the identity on $A_{j}$, for all $j \in \kappa-{i}$, and let $(\pi \restriction A_{i})(a_{i_{1}})=a_{i_{2}}\not=a_{i_{1}}$. Then $\pi$ fixes $E$ pointwise, hence $\pi(f) = f$. So, $f(A_{i})=a_{i_{2}}$ which contradicts the fact that $f$ is a function. Thus $LOC_{n}^{-}$ fails in $\mathcal{M}_{\kappa,n}$. 

{\bf{\underline{(3). If $X\in\{\mathcal{P}_{n}$, $\mathcal{P}_{G,H_{n}}, DT, LT\}$, then $X$ fails in $\mathcal{M}_{\kappa,n}$:}}} Since $AC_{n}$ fails in the model from the arguments of the previous paragraph, $\mathcal{P}_{n}$ fails in the model  by \textbf{Proposition 3.2}. Since $AC_{n}$ fails, $\mathcal{P}_{G,H_{n}}$ fails as well (cf. $\S$1.4). Since in $\mathcal{M}_{\kappa,n}$, the linearly-ordered family $\mathcal{A} = \{A_{\alpha} : \alpha< \kappa\}$ of $n$-element sets does not have a  choice function, $DT$ fails in $\mathcal{M}_{\kappa,n}$ by  [\cite{Tac2019}, \textbf{Theorem 3.1(ii)}]. Since in every Fraenkel--Mostowski model of ZFA, $LT$ implies $AC^{WO}$ (cf.[\cite{Tac2019a}, \textbf{Theorem 4.6(i)}]), LT fails in $\mathcal{M}_{\kappa,n}$ since the well-ordered family $\mathcal{A} = \{A_{\alpha} : \alpha< \kappa\}$ does not have a  choice function.
\end{proof}

\begin{corr}{(ZFA)}
{\em (LOC$_{2}^{-}$ + MC) does not imply $CAC^{\aleph_{0}}_{1}$.}
\end{corr}

\begin{proof} Consider the permutation model $\mathcal{M}_{\kappa,n}$ constructed in \textbf{Theorem 5.3} by letting the infinite well-ordered cardinal number $\kappa$ to be $\aleph_{1}$. Rest follows from \textbf{Theorem 4.5} and the arguments of \textbf{Theorem 5.3(2)}.
\end{proof}

We note that $\mathcal{M}_{\aleph_{0},n}$ is actually equal to the model of [\cite{HHT2012}, proof of \textbf{Theorem 11, Case 1}]; for an argument, one follows in much the same way the ideas of [\cite{HT2020}, \textbf{Remark 2}, \textbf{page 589}]. Following the arguments in the proof of \textbf{Theorem 5.3(3)}, we can also observe that $DT$ and $LT$ fails in the model from [\cite{HT2020}, \textbf{Theorem 10(ii)}]. 
%%%%%%%%%%%%%%%%%%%%%%%%%%%%%%%%%%
\begin{Remark}
We recall two more permutation models where $MC$ holds. 
\begin{itemize}
\item We recall that $MC$ holds in the Second Fraenkel model (labeled as Model $\mathcal{N}_{2}$ in \cite{HR1998}) (cf. \cite{HR1998}). 
\item Fix a prime $p\in \omega$.
We recall the model $\mathcal{N}_{22}(p)$ from [\cite{HT2013}, \textbf{$\S$4.4}]. Let $A$ be the disjoint union of countably many sets of cardinality $p$, i.e., $A = \bigcup_{i\in \omega} A_{i}$ where for each $i \in \omega$, $A_{i} =
\{a_{i,1} ,a_{i,2} ,...,a_{i,p}\}$. Let $\mathcal{G}$ be the group generated by $\{\phi_{i} : i \in \omega\}$ where for each $i \in \omega$,  $\phi_{i}$ is the cycle $(a_{i,1},a_{i,2},...,a_{i,p})$. Let $\mathcal{N}_{22}(p)$ is the model determined by $\mathcal{G}$ and the finite support filter $\mathcal{F}$.
Howard--Tachtsis [\cite{HT2013}, \textbf{Theorem 4.7}] proved that $MC(q)$ holds in $\mathcal{N}_{22}(p)$ for every prime $q\not=p$. 
\end{itemize}

Fix a prime $p\in\omega$. Let $X\in \{CS, CWF, vDCP, LW\}$. Since $MC$ is true in the model of the proof of \textbf{Theorem 5.3}, $\mathcal{N}_{2}$, and $\mathcal{N}_{22}(p)$, $X$ holds in all the mentioned models by \textbf{Remark 5.2} and \textbf{Lemma 2.5}. However, we note that \textbf{Lemma 5.1} can play under certain premises, {\em independently from $MC$}. In particular, the referee communicated to us that $MC$ may fail in the following permutation model $\mathcal{N}$—for example, take $A$ countably infinite, $\mathcal{G}$ the group of all finitary permutations of $A$ and $\mathcal{F}$ the finite support filter; then the resulting permutation model is equal to the basic Fraenkel model, in which $MC$ is false. On the other hand, $X$ is true in the basic Fraenkel model (cf. \cite{Tac2018}, \cite{HST2016}).
We can use \textbf{Lemma 5.1} to see that $X$ holds in the model of the proof of \textbf{Theorem 5.3}, $\mathcal{N}_{2}$, and $\mathcal{N}_{22}(p)$, without invoking $MC$, because all the models are determined by a group $\mathcal{G}$ with the following properties, and the finite support filter $\mathcal{F}$. 
\begin{enumerate}
    \item Every permutation $\phi\in\mathcal{G}$ moves only finitely many atoms in $\mathcal{M}_{\kappa,n}$.
    
    \item We note that $\mathcal{N}_{2}$ was constructed via a group $\mathcal{G}$ such that for all $\phi\in \mathcal{G}$, $\phi^{2} = 1_{A}$. 
    \item We note that $\mathcal{N}_{22}(p)$ was constructed via a group $\mathcal{G}$ such that for all $\phi\in \mathcal{G}$, $\phi^{p} = 1_{A}$. 
\end{enumerate}
\end{Remark}

%%%%%%%%%%%%%%%%%
\section{Conclusion}
\subsection{Synopsis of theorems, propositions, and corollaries}
\begin{itemize}
\item (ZF) $(\forall n\in \omega\backslash\{0,1\})AC_{n}\leftrightarrow \mathcal{P}_{n}$ (cf. [\textbf{$\S$3}, \textbf{Proposition 3.2}]).
\item (ZF) $UT(\aleph_{0},\aleph_{0},\aleph_{0})\rightarrow \mathcal{P}_{lc,c}\rightarrow AC_{\aleph_{0}}^{\aleph_{0}}\rightarrow AC_{fin}^{\omega}\longleftrightarrow \mathcal{P}_{lf,c}$ (cf. [\textbf{$\S$3}, \textbf{Proposition 3.3}, \textbf{Proposition 3.4}]).  
\item (ZF) $AC_{2}\longleftrightarrow \mathcal{P}_{G,H_{2}}$ (cf. [\textbf{$\S$3}, \textbf{Proposition 3.6}]).
\item (ZF) $BPI\longleftrightarrow \mathcal{P}_{G,H_{n}}$ if $n\geq 3$ (cf. [$\S$1.4]). 
\item (ZF)  Let $G$ be a locally finite and connected graph.
\begin{itemize}
    \item {\em $AC_{fin}^{\omega}\rightarrow$  $\mathcal{P}_{G,H}$} (cf. [\textbf{$\S$3}, \textbf{Proposition 3.7}]).
    \item $AC_{fin}^{\omega}\rightarrow$ {\em `If $G=(V_{G},E_{G})$ is a chordal graph, then there is an ordering $<$ of $V_{G}$ such that $\{w < v : \{w,v\} \in E_{G}\}$ is a clique for each $v\in V_{G}$'} (cf. [\textbf{$\S$3}, \textbf{Proposition 3.8}]).
\end{itemize}
\item (ZFA) Let $n \in \omega \backslash \{0, 1\}$. The statement “For every regular $\aleph_{\alpha}$, $CAC^{\aleph_{\alpha}}_{1}$” implies neither $AC_{n}^{-}$ nor “there are no amorphous sets” (cf. [\textbf{$\S$4}, \textbf{Theorem 4.3}]). 
\item (ZF) $CAC^{\aleph_{0}}_{1} \rightarrow PAC^{\aleph_{1}}_{fin}$ (cf. [\textbf{$\S$4}, \textbf{Theorem 4.5}]).
\item (ZF) $DC \not\rightarrow$ {\em $CAC_{1}^{\aleph_{0}}$} (cf. [\textbf{$\S$4}, \textbf{Corollary 4.6}]).
%\item (ZF) CWF $\rightarrow$ KW$_{PO}$ (cf. [$\S$5, \textbf{Observation 5.1}]).
\item (ZFA) Let $n\in\omega$ such that $n=3$ or $n>4$. Then ($LOC_{2}^{-}$ + $MC$) $\not\rightarrow X$, if $X\in \{LOC_{n}^{-}, DT, LT\}$ (cf. [\textbf{$\S$5}, \textbf{Theorem 5.3}]).
\item (ZFA) ($LOC_{2}^{-}$ + $MC$) $\not\rightarrow CAC^{\aleph_{0}}_{1}$ (cf. [\textbf{$\S$5}, \textbf{Corollary 5.4}]).
\end{itemize}

\subsection{Questions and further studies} 
In this paper, we studied the relationship of certain weak choice principles, like $AC_{n}$, $AC_{fin}^{\omega}$, $AC_{\aleph_{0}}^{\aleph_{0}}$, and $UT(\aleph_{0}, \aleph_{0}, \aleph_{0})$, with some weaker formulations of $\mathcal{P}$ (cf. $\S$1.1) in ZF. It would be interesting to see if some other weak choice principle, like $BPI$, is equivalent to some weaker formulation of $\mathcal{P}$ in ZF. 

For a natural number $k<\omega$, we denote  by $\mathcal{Q}_{k}$ the following statement. 
\begin{center}
    {\em `$\chi(E_{G_{1}})=k<\omega$ and $\chi(E_{G_{2}})\geq\omega$ implies $\chi(E_{G_{1}\times G_{2}})=k$.'}
\end{center}
We observed that under $AC_{fin}^{\omega}$ the above statement holds in ZF if $G_{1}$ is a locally finite connected graph (\textbf{cf. Remark 3.9 (2)}). Moreover, we proved that if $X\in \{AC_{3}, AC^{\omega}_{fin}\}$, then  $\mathcal{Q}_{k}\not\rightarrow X$ in ZFA when $k=3$ (\textbf{cf. [\cite{BG2020}, $\S$1]}). We recall the following problem posted in \cite{BG2020}.

\begin{question}{(\textbf{[\cite{BG2020}, Question 5.2]})} {\em If $k>3$, does BPI follow from $\mathcal{Q}_{k}$? Otherwise is there any model of ZF or ZFA, where $\mathcal{Q}_{k}$ holds for $k>3$, but BPI fails?}
\end{question} 

In this direction, we ask the following question.

\begin{question}{\em Is $AC_{2}$ equivalent to $\mathcal{Q}_{3}$? Otherwise is there any model of ZF or ZFA, where $\mathcal{Q}_{3}$ fails, but $AC_{2}$ holds?}
\end{question} 

Fix $k\in \omega\backslash \{0,1,2\}$ and $n\in \omega\backslash\{0,1\}$. We recall that if $X\in  \{\mathcal{P}_{G,H_{n}}, DT, MHT, \mathcal{Q}_{k}\}$, then $BPI$ implies $X$ in ZF (cf. \cite{Haj1985}, \cite{Tac2019}, \cite{HR1998}). It would be interesting to see the interrelationship between the above mentioned implications of $BPI$ in ZF. For instance, we recall the following open problem posted in \cite{Tac2019}.

\begin{question}{(cf. \textbf{[\cite{Tac2019}, $\S$ 4]})} 
{\em Does MHT imply DT?}
\end{question}

In this direction, we ask the following question.
\begin{question}
{\em Fix $k\in \omega\backslash \{0,1,2\}$. Does MHT imply $\mathcal{Q}_{k}$ in ZF?}
\end{question}

We list three more open problems related to $DT$ from \cite{Tac2019}.

\begin{question}{(cf. \textbf{[\cite{Tac2019}, $\S$ 4]})} 
{\em Is there a model of ZFA in which $AC^{LO}$ is true, but $DT$ is false?}
\end{question}

\begin{question}{(cf. \textbf{[\cite{Tac2019}, $\S$ 4]})} 
{\em Does RSL imply DT in ZF?}
\end{question}

\begin{question}{(cf. \textbf{[\cite{Tac2019}, $\S$ 4]})} 
{\em Does DT + $AC_{fin}$ imply BPI?}
\end{question}

Secondly, we studied the relationship of certain weak choice principles, like $MC$, $LOC_{2}^{-}$, $DC$, $AC_{n}^{-}$, and `there are no amorphous sets', with $CAC_{1}^{\aleph_{0}}$ in ZFA. It would be interesting to see the relationship of some other weak choice principles with $CAC_{1}^{\aleph_{0}}$ and $CAC^{\aleph_{0}}$ in ZFA. We also observed that under certain hypotheses on the group $\mathcal{G}$ and the normal filter $\mathcal{F}$, $CWF$ is true in the resulting permutation model $\mathcal{N}$ (cf. \textbf{Lemma 5.1}). The results $MC$ implies $CWF$ in ZFA (cf. \textbf{Remark 5.2}) and $CWF$ implies $LW$ in ZFA (cf. \textbf{Lemma 2.5}) are known.  It would be interesting to see the relationship of $CWF$ with some other weak choice principles in ZFA. We also list two open problems related to $CS$ from \cite{HST2016}.
\begin{question}{(\textbf{[\cite{HST2016}, Problem 5.1]})} 
{\em Is CS equivalent to AC in ZF? If not, which weak choice principles does CS imply? In particular,
does CS imply vDCP?}
\end{question}

\begin{question}{(\textbf{[\cite{HST2016}, Problem 5.3]})} 
{\em Does the Antichain Principle imply CS?}
\end{question}
%%%%%%%%%%%%%%%%%%%%%%%%%
\section{Acknowledgements}
We would like to thank Professor Eleftherios Tachtsis for communicating \textbf{Lemma 4.1} to us in a private conversation. We would like to thank the authors of \cite{HST2016} for communicating the proof of the statement `{\em MC implies CWF in ZFA}' to the reviewer (cf. \textbf{Remark 5.2}) from one of their manuscripts. We would like to thank the reviewer for reading the manuscript carefully and providing several suggestions which improved the exposition of our paper.
%%%%%%%%%%%%%%%%%%%%%%%%

\end{document}